\documentclass{amsart}

\usepackage{amsmath,amssymb,amscd}
\usepackage{hyperref}
\usepackage[all]{xy}

\newtheorem{theorem}{Theorem}[section]
\newtheorem{lemma}[theorem]{Lemma}

\newtheorem{corollary}[theorem]{Corollary}
\newtheorem{proposition}[theorem]{Proposition}

\theoremstyle{definition}
\newtheorem{definition}[theorem]{Definition}
\newtheorem{example}[theorem]{Example}

\theoremstyle{remark}

\numberwithin{equation}{section}

\newcommand{\GL}{\operatorname{GL}}
\newcommand{\SL}{\operatorname{SL}}
\newcommand{\Gr}{\operatorname{Gr}}

\newcommand{\Rep}{\operatorname{Rep}}

\newcommand{\Hom}{\operatorname{Hom}}

\newcommand{\ext}{\operatorname{ext}}

\newcommand{\Coker}{\operatorname{Coker}}

\newcommand{\rank}{\operatorname{rank}}

\newcommand{\innerprod}[1]{\langle#1\rangle}
\newcommand{\Innerprod}[1]{\langle\!\langle#1\rangle\!\rangle}
\newcommand{\sm}[1]{\left(\begin{smallmatrix}#1\end{smallmatrix}\right)}

\newcommand{\mb}[1]{\mathbb{#1}}
\newcommand{\mc}[1]{\mathcal{#1}}

\newcommand{\op}[1]{\operatorname{#1}}
\newcommand{\gha}{{\gamma\hookrightarrow\alpha}}
\newcommand{\gca}{{\gamma,\alpha}}

\newcommand{\Km}{
\vcenter{\xymatrix@C=5ex{
1 \ar@<2ex>[rr]^{a_1}="a"  \ar[rr]|{a_i}="b" \ar@<-2ex>[rr]_{a_m}="c" && 2
\ar @{.} "a";"b" \ar @{.} "b";"c"
}}}

\begin{document}

\title{On Some Quiver Determinantal Varieties}
\author{Jiarui Fei}
\address{Department of Mathematics, University of California, Riverside, CA 92521, USA}
\email{jiarui@ucr.edu}
\thanks{}
\dedicatory{Dedicated to Professor Jerzy Weyman on the Occasion of his Sixtieth Birthday}


\date{}

\begin{abstract} We introduce certain quiver analogue of the determinantal variety.
We study the Kempf-Lascoux-Weyman complex associated to a line bundle on the variety.
In the case of generalized Kronecker quivers, we give a sufficient condition on when the complex resolves a maximal Cohen-Macaulay module supported on the quiver determinantal variety. This allows us to find the set-theoretical defining equations of these varieties.
When the variety has codimension one, the only irreducible polynomial function is a relative tensor invariant.
As a by-product, we find some vanishing condition for the Kronecker coefficients.
In the end, we make a generalization from the quiver setting to the tensor setting.
\end{abstract}

\maketitle
\section*{Introduction}
We work over a field $k$ of characteristic 0. Let $Q$ be some finite quiver with vertex set $Q_0$ and arrow set $Q_1$. For some dimension vector $\alpha$ of $Q$, let $\Rep_\alpha(Q)$ be the space of all $\alpha$-dimensional representations of $Q$. The product of general linear group $\GL_\alpha=\prod_{v\in Q_0} \GL_{\alpha_v}$ acts naturally on $\Rep_\alpha(Q)$.
For another dimension vector $\gamma$, we consider the variety
$$\Rep_{\gha}(Q):=\{M\in\Rep_\alpha(Q)\mid M \text{ has a $\gamma$-dimensional subrepresentation} \}.$$
When $Q$ is the Dynkin $A_2$-quiver, this is a usual determinantal variety. So in this sense, it is a certain quiver generalization of usual determinantal varieties.
Another instance of such varieties is that they appear as {\em exceptional varieties} \cite{Fm} and irreducible components of the {\em null-cone} for the $\GL_\alpha^\sigma$-action on $\Rep_\alpha(Q)$. Here, $\GL_\alpha^\sigma$ is certain codimension one subgroup of $\GL_\alpha$.
In general, the variety $\Rep_{\gha}(Q)$ is highly singular, but it is easy to construct certain Springer-type resolution.

Let $\Gr\sm{\alpha\\ \gamma}$ be the product of Grassmannian varieties $\prod_{v\in Q_0} \Gr \sm{\alpha(v)\\ \gamma(v)}$. Consider
$$Z=\{(L,M)\in \Gr\sm{\alpha\\ \gamma}\times \Rep_\alpha(Q) \mid L \text{ is a subrepresentation of }  M\}.$$
We have the following correspondence, where $p$ is the structure map of a vector bundle and $q$ is the desingularization.
$$\xymatrix{ & Z \ar[dl]_{p} \ar[dr]^{q} & \\
\Gr\sm{\alpha\\ \gamma} && \Rep_\gha(Q)}$$
Moreover, $Z$ can be realized as the total space of some subbundle of the trivial vector bundle $\Gr\sm{\alpha\\ \gamma}\times \Rep_\alpha(Q)$. This allows us to use the {\em Kempf-Lascoux-Weyman's complex} \cite{W} to study the variety $\Rep_{\gha}(Q)$.
The method in \cite{W} reaches its full strength if $\Rep_{\gha}(Q)$ has {\em rational singularities} and $q$ is birational.
Unfortunately, this nice situation rarely occurs in general. To be more precise, when $Q$ is non-Dynkin, for most dimension vectors, the variety $\Rep_{\gha}(Q)$ is not normal. The best situation one can hope is that all higher direct images $\mc{R}^iq_*\mc{O}_Z$ vanish and $q$ is birational, then the KLW-complex is the minimal free resolution of the normalization of $\Rep_{\gha}(Q)$.
However, when $Q$ is wild, for most dimension vectors, some higher direct images $\mc{R}^iq_*\mc{O}_Z$ do not vanish.
We restate the main theorems in \cite{W} in our setting. They are Theorem \ref{T:main}, \ref{T:codim}, and \ref{T:dual}.

It seems hopeless to understand the free resolution of $\Rep_\gha(Q)$ in general, but we still hope to find the defining equations of these varieties, at least set-theoretically. To be more practical, we focus on the case of $m$-arrow Kronecker quivers $K_m$.
For one thing, the sheaf cohomology involved in the KLW-complex can be explicitly computed if we introduce the {\em Kronecker coefficients}.
The Kronecker coefficient $g_{\mu,\nu}^\lambda$ is by definition the structure constant in the tensor product
$$S_\mu\otimes S_\nu = \bigoplus_{\lambda} g_{\mu,\nu}^\lambda S_\lambda,$$
where $S_\lambda$ is the irreducible representation of the symmetric group defined by the partition $\lambda$.
By Schur-Weyl duality, it also appears in
\begin{equation} \label{eq:KC} S^\lambda(V\otimes W) = \sum_{\mu,\nu} g_{\mu,\nu}^\lambda S^\mu(V)\otimes S^\nu(W), \end{equation}
where $S^\lambda$ is the Schur functor corresponding to $\lambda$.
For another thing, when $\gamma=(1,\alpha_2-1)$ our quiver determinantal variety of $K_m$ coincides with the variety constructed from certain $3$-tensor in \cite{BEKS}.
Motivated by some ideas in \cite{BEKS}, we consider the construction in \cite{W} for some line bundle on $\Gr\sm{\alpha\\ \gamma}$. In Proposition \ref{P:Fi} we compute each term of the KLW-complex for any line bundle.
If the complex has no negative degree term, it minimally resolves a module supported on $\Rep_{\gha}(Q)$. This allows us to determine the set-theoretical defining equations of $\Rep_{\gha}(Q)$.

Recall that line bundles on ordinary Grassmannians are indexed by $\mb{Z}$, so for $K_m$ line bundles on $\Gr\sm{\alpha\\ \gamma}$ are parameterized by $\mb{Z}\times\mb{Z}$.
Our main result concerns how to choose an element $\omega$ in $\mb{Z}\times \mb{Z}$ such that the corresponding KLW-complex $F_\bullet^\omega$ has no negative degree term. This is done in Lemma \ref{L:nonnegative}.
We hope to find weights such that the length of $F_\bullet^\omega$ is equal to the codimension of $\Rep_\gha(K_m)$, i.e., $F_\bullet^\omega$ resolves a {\em maximal Cohen-Macaulay} module. This can be easily done by applying the duality theorem to Lemma \ref{L:nonnegative}.
We introduce the notation $\hom_{Q}(\gamma,\beta)$ (resp. $\ext_{Q}(\gamma,\beta)$) to denote the dimension of the space of homomorphisms (resp. extensions) from a general $\gamma$-dimensional representation to a general $\beta$-dimensional representation of $Q$.

\begin{theorem} \label{T:intro1} Let $\beta=\alpha-\gamma$, and assume that $\hom_{K_m}(\gamma,\beta)=0$. If $\omega$ and its dual $\omega^\vee$ satisfy $(\beta_1-w_1)^2+(\gamma_2-w_2)^2<8$ or $(\beta_1-w_1)^2+(\gamma_2-w_2)^2=8$ with any of the following: $(1).\ \beta_1\neq \gamma_2$, $(2).\ w_1\neq w_2$, and $(3).\ w_1+w_2>m-3$,
then the complex $F_\bullet^\omega$ resolves a maximal Cohen-Macaulay module supported on $\Rep_\gha(K_m)$.
\end{theorem}

However, the existence of such a weight is not guaranteed by the theorem. The result is sharp only in some cases. A sharp result would depend on a good understanding on the Kronecker coefficients. The other way around, we can actually deduce some interesting vanishing conditions on the Kronecker coefficients.
We denote by $P(s,q,t,w)$ the set of all partitions $\lambda$ with at most $s$ parts satisfying $\lambda_t\geqslant q+t+w$ and $\lambda_{t+1}\leqslant t+w$.

\begin{theorem} \label{T:intro2} Let $w_1,w_2$ be two non-positive integers. \begin{enumerate}
\item For $\mu\in P(\gamma_1,\beta_1,t_1,w_1), \nu\in P(\beta_2,\gamma_2,t_2,w_2)$ with $|\mu|=|\nu|>\beta_1 t_1+\gamma_2 t_2+\ext_{K_m}(\gamma,\beta)$,
we have that $g_{\mu,\nu}^\lambda$ vanishes if $\lambda_1\leqslant m.$
\item For $\mu\in P(\gamma_1,\beta_1,t_1,m\beta_2-\alpha_1-w_1), \nu\in P(\beta_2,\gamma_2,t_2,m\gamma_1-\alpha_2-w_2)$ with $|\mu|=|\nu|<\beta_1 t_1+\gamma_2 t_2-\hom_{K_m}(\gamma,\beta)$,
we have that $g_{\mu,\nu}^\lambda$ vanishes if $\lambda_1\leqslant m.$
\end{enumerate}
\end{theorem}

When the variety $\Rep_\gha(Q)$ has codimension one in $\Rep_\alpha(Q)$, the single irreducible defining polynomial $\Delta_{\alpha,m}^\gamma$ is a relative tensor invariant.
It can be computed by the determinant of the complex (Proposition \ref{P:det}). When the complex has length two, we get a determinantal formula for $\Delta_{\alpha,m}^\gamma$.
We find all such polynomials for $2\leqslant m, \alpha_1,\alpha_2\leqslant 5$ (Example \ref{ex:rigid}, \ref{ex:codim1}). It is quite surprising that we can always find a weight such that the differential is linear, i.e., of degree one.

Finally, we make one possible generalization from the quiver setting to the tensor setting in the last section.
We consider an analogous quotient bundle $\mc{E}$ such that the corresponding subbundle desingularizes some variety $R_\gca$.
Proposition \ref{P:TFi}, Corollary \ref{C:Kron}, and Proposition \ref{P:TMCM} are analogues of Proposition \ref{P:Fi}, Theorem \ref{T:intro2}, and Theorem \ref{T:intro1}. When $R_\gca$ has codimension one, it also corresponds to a relative tensor invariant. We also find all such invariants for $2\leqslant \alpha_i\leqslant 5$.

\section{Review of Vector Bundles on Grassmannians} \label{S:Bott}
Let $\Gr\sm{r\\s}$ be the Grassmannian variety parameterizing $s$-dimensional subspace in $R=k^r$. Let $\mc{S}$ and $\mc{Q}$ be the universal sub- and quotient bundles on $\Gr\sm{r\\s}$
$$0\to \mc{S}\to \Gr\sm{r\\s}\times R\to \mc{Q} \to 0.$$

Given a permutation $\sigma$, we define the {\em length} of $\sigma$ to be $\ell(\sigma)=\#\{i<j\mid\sigma(i)>\sigma(j)\}$. Also, define $\rho=(r-1,r-2,\dots,1,0)$. Given a sequence of integers $\alpha$, we define $\sigma\circ\alpha = \sigma(\alpha+\rho)-\rho$.
\begin{theorem}[Borel-Weil-Bott] Let $\mu,\nu$ be two partitions, and set $\lambda=(\mu,\nu).$ Then exactly one of the following two situations occur.
\begin{enumerate} \item There exists $\sigma\neq \op{id}$ such that $\sigma \circ \lambda = \lambda.$ Then all cohomologies of $S^\mu \mc{Q} \otimes S^\nu \mc{S}$ vanish.
\item There is a (unique) $\sigma$ such that $\eta=\sigma\circ\lambda$ is a weakly decreasing sequence. Then
$$H^{\ell(\sigma)}(\Gr\sm{r\\ s}; S^\mu \mc{Q} \otimes S^\nu \mc{S})=S^\eta R$$
and all other cohomologies vanish.
\end{enumerate}
\end{theorem}

One important case to us is the vector bundle $S^\mu \mc{S}\otimes \det^w \mc{Q}$ or $S^\nu \mc{Q}^*\otimes \det^w \mc{S}^*$.
To apply Bott's algorithm, we consider $(w^q,\mu)+\rho=(r-1+w,\dots,r-q+w,\mu_1+r-q-1,\dots,\mu_s),$ where $q=\rank \mc{Q} = r-s$.
To produce nontrivial cohomology, $(w^q,\mu)+\rho$ cannot have any repetition.
Let $t$ be the biggest number such that $\mu_t+r-q-t>r-1+w$, then $\mu_{t+1}+r-q-t-1<r-q+w$.
In terms of $\mu$, this means that
\begin{equation} \label{eq:P} \mu_t\geqslant q+t+w,\quad \mu_{t+1}\leqslant t+w.
\end{equation}
We introduce the notation $P(s,q,t,w)$ to denote all partitions with at most $s$ parts satisfying \eqref{eq:P}.
Let $\sigma(t)$ be the permutation that moves $\mu_1+r-q-1,\dots,\mu_t+r-q-t$ in front of $r-1+w,\dots,r-q+w$, then clearly $\ell(\sigma(t))=qt$ and
$$\sigma(t)\circ(w^q,\mu)=(\mu_1-q,\dots,\mu_t-q,(t+w)^q,\mu_{t+1},\dots,\mu_s).$$

So we computed the first part of the following corollary (see also \cite[p.162]{W}), and the second half is similar.
\begin{corollary} \label{C:Bott} $H^\bullet(\Gr\sm{r\\ s}; S^\mu \mc{S}\otimes \det^w \mc{Q})$ is zero unless $\mu\in P(s,q,t,w)$. In that case,
all cohomology groups vanish except that
$$H^{qt}(\Gr\sm{r\\ s}; S^\mu \mc{S}\otimes {\det}^w \mc{Q})=S^{\sigma(t)\circ(w^q,\mu)}R.$$
Similarly $H^{\bullet}(\Gr\sm{r\\ s}; S^\nu \mc{Q}^*\otimes {\det}^w \mc{S}^*)$ is zero unless $\mu\in P(q,s,t',w)$. In that case,
all cohomology groups vanish except that
$$H^{st'}(\Gr\sm{r\\ s}; S^\nu \mc{Q}^*\otimes {\det}^w \mc{S}^*)=S^{\sigma(t')\circ(w^s,\nu)}R^*.$$
\end{corollary}

\section{Some Quiver Determinantal Varieties} \label{S:Qdet}
Fix a finite quiver $Q=(Q_0,Q_1)$ and two dimension vectors $\alpha$ and $\gamma$. In what follows, we always assume $\beta=\alpha-\gamma$. We define $\Gr\sm{\alpha\\ \gamma}=\prod_{v\in Q_0} \Gr \sm{\alpha(v)\\ \gamma(v)}$. Given any vector bundle $\mc{V}$ on $\Gr \sm{\alpha(v)\\ \gamma(v)}$, we can pull it back to $\Gr\sm{\alpha\\ \gamma}$ via the projection $\pi_v$. To simplify our notation, we will write $\mc{V}_1\otimes \mc{V}_2$ instead of $\mc{V}_1\boxtimes \mc{V}_2:=\pi^*(\mc{V}_1)\otimes \pi^*(\mc{V}_2)$
if no potential confusion can arise.

The space of all $\alpha$-dimensional representations of $Q$ is
$$\Rep_\alpha(Q):=\bigoplus_{a\in Q_1}\Hom(k^{\alpha(ta)},k^{\alpha(ha)}),$$
where $ta$ and $ha$ are the tail and head of $a$.
Let $\mc{S}_v$ and $\mc{Q}_v$ be the universal sub- and quotient bundles on $\Gr\sm{\alpha(v)\\ \gamma(v)}$.
We denote the vector space $k^{\alpha(v)}$ by $R_v$, and the corresponding trivial bundle by $\mc{R}_v$.
Let $\mc{E}$ be the vector bundle on $\Gr\sm{\alpha\\ \gamma}$ defined by
$$\mc{E}:=\bigoplus_{a\in Q_1} \mc{H}om(\mc{S}_{ta},\mc{Q}_{ha}).$$
Consider the vector bundle epimorphism $\Gr\sm{\alpha\\ \gamma}\times \Rep_\alpha(Q) \to \mc{E}$
induced by tensoring $\mc{R}_{ta}^* \twoheadrightarrow \mc{S}_{ta}^*$ and $\mc{R}_{ha} \twoheadrightarrow \mc{Q}_{ha}$.
Fibrewise it sends a representation $M$ over $S\in \Gr\sm{\alpha\\ \gamma}$ to $\bigoplus_{a\in Q_1}\Hom(S_{ta},M_{ha}/S_{ha})$ by restriction and projection.
The kernel is a vector bundle denoted by $\mc{Z}$:
\begin{equation}\label{eq:exactZ} 0\to \mc{Z} \to \Gr\sm{\alpha\\ \gamma}\times \Rep_\alpha(Q) \to \mc{E}\to 0.\end{equation}
It is clear that the total space of $\mc{Z}$ is the following variety
$$Z=\{(L,M)\in \Gr\sm{\alpha\\ \gamma}\times \Rep_\alpha(Q) \mid L \text{ is a subrepresentation of }  M\}.$$
Consider the projections to the first and the second factors.
$$\xymatrix{ & Z \ar[dl]_{p} \ar[dr]^{q} & \\
\Gr\sm{\alpha\\ \gamma} && \Rep_\alpha(Q)}$$

We proved that (see \cite[Section 3]{S})
\begin{lemma}
The map $p: Z\to \Gr\sm{\alpha\\ \gamma}$ is the vector bundle with fibre
$$\bigoplus_{a\in Q_1} \Hom(k^{\gamma(ta)},k^{\gamma(ha)})\oplus \Hom(k^{\beta(ta)},k^{\alpha(ha)}).$$
In particular, $Z$ is smooth irreducible of dimension equal to $\dim\Rep_\alpha(Q)+\innerprod{\gamma,\beta}$.
\end{lemma}
Here, $\innerprod{-,-}:\mb{Z}^{Q_0}\times \mb{Z}^{Q_0}\to \mb{Z}$ is the Euler form of $Q$. By definition it is given by $\innerprod{-,-}:=\innerprod{-,-}_0-\innerprod{-,-}_1$, where $\innerprod{-,-}_0$ is the usual dot product and $\innerprod{\gamma,\beta}_1=\sum_{a\in Q_1} \gamma(ta)\beta(ha)$.
It is well-known that $\innerprod{\gamma,\beta}=\hom_Q(\gamma,\beta)-\ext_Q(\gamma,\beta)$.
Schofield discovered a recursive algorithm to compute $\ext_Q(\gamma,\beta)$ in \cite{S}.

\begin{definition} We define certain quiver analogue of determinantal varieties
$$\Rep_{\gha}(Q):=\{M\in\Rep_\alpha(Q)\mid M \text{ has a $\gamma$-dimensional subrepresentation} \}.$$
Since $Z$ is integral and $q$ is projective, the scheme-theoretical image $q(Z)$ is integral and closed, and hence equal to $\Rep_\gha(Q)$.
We always assume that $\Rep_{\gha}(Q)$ is strictly contained in $\Rep_\alpha(Q)$.
Note that when $Q$ is the $A_2$-quiver, such a variety is a usual determinantal variety.
\end{definition}

From now on, we will use $q$ to denote the map $q:Z\to \Rep_{\gha}(Q)$.
\begin{lemma}\cite{S}  The dimension of a general fibre of $q$ is equal to $\hom_Q(\gamma,\beta)$.
The codimension of $\Rep_{\gha}(Q)$ in $\Rep_\alpha(Q)$ is equal to $\ext_Q(\gamma,\beta)$.
\end{lemma}
So we always assume that $\ext_Q(\gamma,\beta)>0$. Moreover, $\hom_Q(\gamma,\beta)=0$ is a necessary condition for $q$ being birational. Note that the combination of $\hom_Q(\gamma,\beta)=0$ and $\ext_Q(\gamma,\beta)>0$ is equivalent to the condition
\begin{equation} \label{eq:homext} \hom_Q(\gamma,\beta)=0,\quad\text{and}\quad \innerprod{\gamma,\beta}<0.\end{equation}

It is well-known that in characteristic 0, rational bijective implies birational. We are going to give a numerical criterion for $q$ to be a birational isomorphism.
By Bertini's theorem the general fibre of $q$ is reduced, so let us assume the opposite that the general fibre of $q$ contains more than one representation. Then a general representation in $\Rep_{\gha}(Q)$ have at least two $\gamma$-dimensional subrepresentations. There exists a dimension vector $\delta$ such that a general representation $M\in \Rep_{\gha}(Q)$ has subrepresentations $L_1, L_2$ of $M$ such that $\dim(L_1\cap L_2)=\delta$.
Consider the incidence varieties
\begin{align*}
&\op{Inc}\sm{\alpha \\ \gamma\cap \gamma=\delta}=\{(V_1,V_2)\in\Gr \sm{\alpha\\ \gamma} \times \Gr \sm{\alpha\\ \gamma}\mid \dim(V_1\cap V_2)=\delta \},\\
& Z_\delta =\{(M,L_1,L_2)\in \op{Inc}\sm{\alpha \\ \gamma\cap \gamma=\delta}\times \Rep_\alpha(Q) \mid L_1, L_2 \text{ are subrepresentations of } M\}.
\end{align*}
The first one is a smooth irreducible (non-closed) subvariety of $\Gr \sm{\alpha\\ \gamma} \times \Gr \sm{\alpha\\ \gamma}$ of codimension equal to $\innerprod{\delta,\beta-\gamma+\delta}_0$.
The following lemma is straightforward.

\begin{lemma} \label{L:Zdelta} $Z_\delta$ is a vector bundle over $\op{Inc}\sm{\alpha \\ \gamma\cap \gamma=\delta}$ with fibre
$$\Hom(k^\delta,k^\delta)\oplus2\Hom(k^{\gamma-\delta},k^{\gamma})\oplus\Hom(k^{\beta-\gamma+\delta},k^\alpha).$$
In particular, $Z_\delta$ is smooth and irreducible with dimension equal to $$\dim\Rep_\alpha(Q)+2\innerprod{\gamma,\beta}-\innerprod{\delta,\beta-\gamma+\delta}.$$
\end{lemma}

Now we assume that $\hom_Q(\gamma,\beta)=0$, so $\dim \Rep_{\gha}(Q)=\dim\Rep_\alpha(Q)+\innerprod{\gamma,\beta}$.
Let $q_\delta$ be the projection from $Z_\delta\to \Rep_\alpha(Q)$. By our assumption, we have that $\overline{q_\delta(Z_\delta)}=\Rep_{\gha}(Q)$. In particular, $\dim Z_\delta \geqslant \dim\Rep_\alpha(Q)+\innerprod{\gamma,\beta}$.
So by Lemma \ref{L:Zdelta}, $\innerprod{\delta,\beta-\gamma+\delta}\leqslant \innerprod{\gamma,\beta}$.

Moreover, every representation in $\Rep_{\gha}(Q)$ has a $(2\gamma-\delta)$-dimensional subrepresentation.
So $\Rep_{2\gamma-\delta\hookrightarrow \alpha}(Q)\supseteq \Rep_{\gha}(Q)$, and thus
$\ext_Q(2\gamma-\delta,\beta-\gamma+\delta)\leqslant-\innerprod{\gamma,\beta}$. Therefore
$\innerprod{2\gamma-\delta,\beta-\gamma+\delta}\geqslant\innerprod{\gamma,\beta}$. So we proved

\begin{proposition} \label{P:birational} Assume that $\hom_Q(\gamma,\beta)=0$.
If for any $\delta\precneqq \gamma$ with $2\gamma-\delta\preceq\alpha$,
$$\text{either} \quad \innerprod{2\gamma-\delta,\beta-\gamma+\delta}<\innerprod{\gamma,\beta} \quad \text{ or }\quad \innerprod{\gamma,\beta} < \innerprod{\delta,\beta-\gamma+\delta},$$
then $q$ is a birational isomorphism.
Here, $\preceq$ is the relation defined by $\delta\prec \gamma$ if and only if $\gamma-\delta\in (\mb{Z}_{\geq 0})^{Q_0}$.
\end{proposition}

\section{Main Construction} \label{S:Main}
In this section, we are going to construct finite free resolution of $q_*(\mc{O}_Z)$.
We note that $q_*\mc{O}_Z$ is finite over $\Rep_{\gha}(Q)$ \cite[Corollary III.11.5]{Ha}.
We are in the situation of the Basic Theorem of \cite{W}
\begin{equation*}
\xymatrix@R=5ex@C=7ex{
Z\ar[dd]_{q} \ar@/^/[drr]^{p} \ar@{_{(}->}[dr]\\
 & \Gr\sm{\alpha \\ \gamma}\times \Rep_\alpha(Q) \ar[d]_{\pi} \ar[r]&  \Gr\sm{\alpha \\ \gamma} \\
\Rep_\gha(Q) \ar@{^{(}->}[r] & \Rep_\alpha(Q) }
\end{equation*}

We denote $GR:=\Gr\sm{\alpha\\ \gamma}\times \Rep_\alpha(Q)$. Consider the locally free resolution of the sheaf $\mc{O}_Z$ as an $\mc{O}_{GR}$-module given by the Koszul complex \cite[Lemma 5.1.1.a]{W}
$$\mc{K}: 0\to \bigwedge^{\innerprod{\gamma,\beta}_1} p^* \mc{E}^* \to \cdots \to \bigwedge^2 p^* \mc{E}^* \to p^*\mc{E}^*\to \mc{O}_{GR}.$$
It turns out that the derived pushforward of this complex by $\pi$ is isomorphic to a complex $F_\bullet$
whose $i$th-component is given by \cite{W}
$$F_i=\bigoplus_{j\geqslant 0} H^j(\Gr\sm{\alpha\\ \gamma}; \bigwedge^{i+j} \mc{E}^*)\otimes A(-i-j),$$
where $A=k[\Rep_\alpha(Q)]$ is the coordinate ring of $\Rep_\alpha(Q)$.
We will compute each $F_i$ for Kronecker quivers in Section \ref{S:Kron}.

We know from \cite[Theorem 5.1.2]{W} that there exist minimal differentials $d_i: F_i\to F_{i-1}$ of degree $0$ such that $F$ is a complex of free graded $A$-modules with $H_{-i}(F_\bullet)=\mc{R}^iq_*\mc{O}_Z$. In particular $F_\bullet$ is exact in positive degree. Moreover, by \cite[Theorem 5.4.1]{W} all differentials can be made $G$-equivariant, where $G:=\GL_\alpha\times \prod_{(u,v)\in Q_0^2} \GL(R_{uv})$, and $R_{uv}=kQ(u,v)$ is the vector space spanned by arrows from $u$ to $v$.
It follows that

\begin{theorem} \label{T:main} Assume that $\mc{R}^i q_*\mc{O}_Z=0$ for $i>0$. Then the $G$-equivariant complex $F_\bullet$ is a minimal free resolution of $q_*\mc{O}_Z$.

If $q$ is a birational isomorphism, $q_*\mc{O}_Z$ is the normalization of $\Rep_{\gha}(Q)$. In particular, the normalization has rational singularities, and hence is Cohen-Macaulay.
\end{theorem}

In general, $\mc{R}^i q_*\mc{O}_Z$ fails to vanish for $i>0$ (see Example \ref{ex:codim1}). However, there are some known special cases. A result of Sutar \cite{Su} says that this holds for all (extended) Dynkin quivers with source-sink orientation. In fact, we conjecture that the condition on the orientation is unnecessary. We also conjecture that this result is sharp in the sense that for any wild quiver, there exist some $\gamma,\alpha$ such that $\mc{R}^1 q_*\mc{O}_Z\neq 0$. We will see that such examples already appear in the simplest wild quiver without oriented cycles, namely the $3$-arrow Kronecker quiver.

According to \cite[Theorem 5.1.3.c]{W}, if $q$ is birational, $\mc{R}^i q_*\mc{O}_Z=0$ for $i>0$, and $F_0=A$, then $\Rep_{\gha}(Q)$ is normal. There are only few known cases, e.g., when the quiver is Dynkin with source-sink orientation \cite{Su}. We also believe that the condition on the orientation can be dropped. We have found for each extended-Dynkin type quiver $Q$, some non-normal $\Rep_{\gha}(Q)$. For such an example for the $2$-arrow Kronecker quiver, see Example \ref{ex:rigid}.

Even if $q$ fails to be birational, the complex $F_\bullet$ still contains some information on $\Rep_{\gha}(Q)$. We restate \cite[Theorem 5.1.6]{W} in our setting

\begin{theorem} \label{T:codim} {\ }\begin{enumerate}
\item $\op{codim} \Rep_{\gha}(Q)=\ext_Q(\gamma,\beta)=\max\{i\mid F_i\neq 0\}$.
\item Assume that $r=-\innerprod{\gamma,\beta}>0$, then
$$\deg(q)\deg(\Rep_\gha(Q))=\sum_{i,j}(-1)^{i+r}\frac{(i+j)^r}{r!}h^j\big(\Gr\sm{\alpha \\ \gamma},\bigwedge^{i+j}\mc{E}^*\big),$$
where $h^j(-)$ is $\dim H^j(-)$, and by definition $\deg(q)$ is $0$ if $\hom_Q(\gamma,\beta)>0$.
\end{enumerate}
\end{theorem}



The complex $F_\bullet$ can be twisted by any vector bundle on $\Gr\sm{\alpha \\ \gamma}$:
\begin{equation} \label{eq:FV}
F_i(\mc{V})=\bigoplus_{j\geqslant 0} H^j(\Gr\sm{\alpha\\ \gamma}; \bigwedge^{i+j} \mc{E}^* \otimes \mc{V})\otimes A(-i-j).
\end{equation}
The twisted complex $F_\bullet(\mc{V})$ can also be equipped with minimal differentials and $G$-equivariant structure.
We also state \cite[Theorem 5.1.4]{W} in our setting. We will use $[-]$ for shifting homological degree, i.e., $F[i]_j=F_{i+j}$.

\begin{theorem} \label{T:dual} Let $\mc{V}$ be a vector bundle $\mc{V}$ on $\Gr\sm{\alpha \\ \gamma}$, and define the dual bundle
$\mc{V}^\vee = \omega \otimes \bigwedge^{\innerprod{\gamma,\beta}_{1}} \mc{E} \otimes \mc{V}^*$, where $\omega$ is the canonical bundle on $\Gr\sm{\alpha \\ \gamma}$. Then
$$F(\mc{V}^\vee)_\bullet = F(\mc{V})_\bullet^*[\innerprod{\gamma,\beta}].$$
\end{theorem}

\section{The Case of Kronecker Quivers} \label{S:Kron}
Let us consider the case when $Q$ is the $m$-arrow Kronecker quiver $K_m$.
$$\Km$$
If $\gamma_1=\alpha_1$, then $\Rep_{\gha}(K_m)$ is isomorphic to the usual determinantal variety of rank $\gamma_2$ maps from $k^{m\alpha_1}$ to $k^{\alpha_2}$. We also have the dual situation if $\gamma_2=0$.

We consider certain twisted version of the complex $F_\bullet$. For fixed {\em weight} $\omega=(\omega_1;\omega_2)\in \mathbb{Z}^{\beta_1}\times \mathbb{Z}^{\gamma_2}$, we put the vector bundle $S^{\omega_1}\mc{Q}_1 \otimes S^{\omega_2}\mc{S}_2^*$ in place of $\mc{V}$ in \eqref{eq:FV}.
The $i$th term of the twisted complex $F_\bullet^\omega$ is
$$F_i^\omega =\bigoplus_{j\geqslant 0} H^j(\Gr\sm{\alpha\\ \gamma}; \bigwedge^{i+j} \mc{E}^* \otimes S^{\omega_1}\mc{Q}_1 \otimes S^{\omega_2}\mc{S}_2^*)\otimes A(-i-j).$$
By \cite[Theorem 5.1.2.b, 5.1.3.a]{W}, if $F_\bullet^\omega$ has no negative degree terms, then it resolves a module supported on $\Rep_\gha(Q)$. We denote this module by $M_{\gamma,\alpha}^\omega$.
If $\omega=(w_1^{\beta_1};w_2^{\gamma_2})$ for $w_1,w_2\in\mathbb{Z}$, then the vector bundle $S^{\omega_1}\mc{Q}_1 \otimes S^{\omega_2}\mc{S}_2^*$ is a line bundle.
We call such a weight a {\em line weight}, and simply write $(w_1;w_2)$.
Note that we get all line bundles on $\Gr\sm{\alpha\\ \gamma}$ this way because the Picard group of any ordinary Grassmannian is $\mb{Z}$.
In the proposition below, we use $\lambda'$ to denote the conjugate partition of $\lambda$.

\begin{proposition} \label{P:Fi} The $i$th term of the complex $F^\omega_\bullet$ is given by
\begin{align*}\bigoplus_{\substack{|\lambda|=|\mu|=|\nu|\\=i+\ell(\sigma_1)+\ell(\sigma_2)}} S^{\sigma_1(\omega_1,\mu)}R_1 \otimes S^{\sigma_2(\omega_2,\nu)}R_2^* \otimes \big(g_{\mu,\nu}^{\lambda} S^{\lambda'}R_{12}\big) \otimes A(-|\lambda|).
\end{align*}
In particular, if $\omega=(w_1;w_2)$ is a line weight, then $F^\omega_i$ is given by
\begin{align*}\bigoplus_{\substack{0\leqslant t_1\leqslant \gamma_1,\\ 0\leqslant t_2\leqslant \beta_2}}
\bigoplus_{\substack{\mu\in P(\gamma_1,\beta_1,t_1,w_1),\nu\in P(\beta_2,\gamma_2,t_2,w_2) \\ |\lambda|=|\mu|=|\nu|=i+\beta_1 t_1+\gamma_2 t_2}}
S^{\mu^\circ}R_1  \otimes S^{\nu^\circ}R_2^*  \otimes \big(g_{\mu,\nu}^{\lambda} S^{\lambda'}R_{12}\big) \otimes A(-|\lambda|),
\end{align*}
where \begin{align}
\label{eq:circ1} \mu^\circ:&= \sigma(t_1)\circ (w^{\beta_1},\mu)=(\mu_1-\beta_1,\dots,\mu_{t_1}-\beta_1,(t_1+w_1)^{\beta_1},\mu_{t_1+1},\dots,\mu_{\gamma_1}),\\
\label{eq:circ2} \nu^\circ:&= \sigma(t_2)\circ (w^{\gamma_2},\nu)=(\nu_1-\gamma_2,\dots,\nu_{t_2}-\gamma_2,(t_2+w_2)^{\gamma_2},\nu_{t_2+1},\dots,\nu_{\beta_2}).
\end{align}
Assume that $F_\bullet^\omega$ has no negative degree terms. Then the annihilator of $M_{\gamma,\alpha}^\omega$ is the prime ideal defining $\Rep_\gha(Q)$, and the maximal minors of $d_1:F_1^\omega\to F_0^\omega$ defines $\Rep_\gha(Q)$ set-theoretically.
\end{proposition}

\begin{proof}
\begin{align*} \bigwedge^{n} \mc{E}^*
&= \bigwedge^{n} \Big( \mc{S}_{1}\otimes \mc{Q}_{2}^* \otimes R_{12} \Big) \notag = \bigoplus_{|\lambda|=n}  S^{\lambda}(\mc{S}_{1} \otimes \mc{Q}_{2}^*)\otimes S^{\lambda'}R_{12}, & \text{(Cauchy formula)}\\
&= \bigoplus_{|\lambda|=|\mu|=|\nu|=n}  \big(g_{\mu,\nu}^{\lambda} S^{\lambda'}R_{12}\big) \otimes
  S^{\mu}\mc{S}_{1} \otimes S^{\nu}\mc{Q}_{2}^*.  & \eqref{eq:KC}
\end{align*}

\begin{align*}
F_i^\omega &=\bigoplus_{j\geqslant 0} H^j(\Gr\sm{\alpha\\ \gamma}; \bigwedge^{i+j} \mc{E}^* \otimes S^{\omega_1}\mc{Q}_1 \otimes S^{\omega_2}\mc{S}_2^*)\otimes A(-i-j), \\
&=\bigoplus_{\substack{j\geqslant 0\\|\lambda|=|\mu|=|\nu|=i+j}}
H^j\big(\Gr\sm{\alpha\\ \gamma};  (S^{\omega_1}\mc{Q}_1 \otimes S^{\mu}\mc{S}_{1}) \otimes (S^{\omega_2}\mc{S}_2^*\otimes S^{\nu}\mc{Q}_{2}^*) \otimes (g_{\mu,\nu}^{\lambda} S^{\lambda'}R_{12}) \big) \otimes A(-|\lambda|), \\
\intertext{Since $g_{\mu,\nu}^{\lambda} S^{\lambda'}R_{12}$ is just a vector space, we can pull it out}
&=\bigoplus_{\substack{j\geqslant 0\\|\lambda|=|\mu|=|\nu|=i+j}}
H^j\big(\Gr\sm{\alpha\\ \gamma};  (S^{\omega_1}\mc{Q}_1 \otimes S^{\mu}\mc{S}_{1}) \otimes (S^{\omega_2}\mc{S}_2^*\otimes S^{\nu}\mc{Q}_{2}^*) \big)\otimes (g_{\mu,\nu}^{\lambda} S^{\lambda'}R_{12}) \otimes A(-|\lambda|), \\
& =\bigoplus_{\substack{|\lambda|=|\mu|=|\nu|\\=i+\ell(\sigma_1)+\ell(\sigma_2)}} S^{\sigma_1(\omega_1,\mu)}R_1 \otimes S^{\sigma_2(\omega_2,\nu)}R_2^* \otimes \big(g_{\mu,\nu}^{\lambda} S^{\lambda'}R_{12}\big) \otimes A(-|\lambda|). \quad\text{(K\"{u}nneth formula)}
\end{align*}
The statement for line weights follows from Corollary \ref{C:Bott}. The statement about maximal minors follows from \cite[Proposition 20.7]{E}.
\end{proof}

There is an obvious symmetry from the formula of $F_i^\omega$. If we set $\gamma'=(\beta_2,\beta_1),\beta'=(\gamma_2,\gamma_1)$, and $\omega'=(\omega_2,\omega_1)$, then we essentially get the same complex.

\begin{definition} A weight $\omega=(\omega_1;\omega_2)$ is called {\em Cohen-Macaulay} if $F_\bullet^\omega$ has no negative degree term and the length of $F_\bullet^\omega$ is the codimension of $\Rep_{\gha}(Q)$, i.e., $M_{\gamma,\alpha}^{\omega}(Q)$ is maximal Cohen-Macaulay.
\end{definition}

\begin{lemma} \label{L:nonnegative} $F_\bullet^\omega$ has no term in negative degree if the line weight $\omega=(w_1;w_2)$ satisfies $(\beta_1-w_1)^2+(\gamma_2-w_2)^2<8$ or $(\beta_1-w_1)^2+(\gamma_2-w_2)^2=8$ with any of the following: $(1).\ \beta_1\neq \gamma_2$, $(2).\ w_1\neq w_2$, and $(3).\ w_1+w_2>m-3$.
\end{lemma}

\begin{proof} We see from Proposition \ref{P:Fi} that $\mu\in P(\gamma_1,\beta_1,t_1,w_1),\nu\in P(\beta_2,\gamma_2,t_2,w_2)$, so
$|\mu|\geqslant (\beta_1+t_1+w_1)t_1, |\nu|\geqslant (\gamma_2+t_2+w_2)t_2$.
It is easy to see that a necessary condition for $F_{-1}^\omega$ nonvanishing is that
\begin{equation} \label{eq:ineq} \begin{cases}
-1+\gamma_2t_2 \geqslant t_1(t_1+w_1) \\-1+\beta_1t_1 \geqslant t_2(t_2+w_2) \end{cases}
\quad \text{for $0\leqslant t_1\leqslant \gamma_1,0\leqslant t_2\leqslant \beta_2$.}
\end{equation}
So if $f(t_1,t_2)=t_1^2+(w_1-\beta_1)t_1+t_2^2+(w_2-\gamma_2)t_2+2>0$, then $F_\bullet^\omega$ has no negative degree term.
Calculus tells us that $f$ has a global minimum $2-\frac{1}{4}(\beta_1-w_1)^2-\frac{1}{4}(\gamma_2-w_2)^2$. The condition (1) follows.

If $(\beta_1-w_1)^2+(\gamma_2-w_2)^2=8$, then $\beta_1-w_1=\gamma_2-w_2=2$. It is clear that $f>0$ unless $t_1=t_2=1$.
If $t_1=t_2=1$, then it follows from \eqref{eq:ineq} that $F_{-1}$ vanishes unless $\beta_1=\gamma_2$ and $w_1=w_2$.
Since $t_1=t_2=1$, we can effectively apply Littlewood and Murnaghan's inequality on Kronecker coefficients, which implies that
$\lambda_1\geqslant \mu_1+\nu_1-|\lambda|\geqslant(\beta_1+1+w_1)+(\gamma_2+1+w_2)-(-1+\beta_1+\gamma_2)=3+w_1+w_2$. So if $3+w_1+w_2>m$, then $S^{\lambda'}(k^m)$ has to vanish.
\end{proof}

Now for each weight $\omega=(\omega_1;\omega_2)$, we introduce the dual weight $\omega^\vee=(m\beta_2-\alpha_1-\omega_1;m\gamma_1-\alpha_2-\omega_2)$.
We justify this definition as follows. Consider the dual vector bundle $\mc{V}^\vee = \omega \otimes \bigwedge^{\innerprod{\gamma,\beta}_{1}} \mc{E} \otimes \mc{V}^*$.
The canonical bundle of $\Gr\sm{\alpha\\ \gamma}$ is
$$\omega=\bigotimes_{v=1,2}(\bigwedge^{\gamma_v}\mc{S}_v)^{\otimes \beta_v} \otimes (\bigwedge^{\beta_v}\mc{Q}_v^*)^{\otimes \gamma_v},$$
and
$$\bigwedge^{\innerprod{\gamma,\beta}_{1}} \mc{E} = (\bigwedge^{\gamma_1}\mc{S}_{1}^*)^{\otimes m\beta_2} \otimes (\bigwedge^{\beta_2}\mc{Q}_{2})^{\otimes m\gamma_1} \otimes (\bigwedge^{m}R_{12})^{\otimes \gamma_1\beta_2}.$$
So \begin{align*}
\mc{V}^\vee &\cong (\bigwedge^{\beta_1} \mc{Q}_1)^{\otimes (m\beta_2-\alpha_1)}  \otimes (\bigwedge^{\gamma_2} \mc{S}_2^*)^{\otimes (m\gamma_1-\alpha_2)} \otimes \mc{V}^*,
\end{align*}
and hence by Theorem \ref{T:dual}
$$F^{\omega^\vee}_\bullet=(F^{\omega}_\bullet)^*[\innerprod{\gamma,\beta}].$$
Then it follows from Lemma \ref{L:nonnegative} that

\begin{lemma} If the dual of a line weight $\omega$ satisfies the condition in Lemma \ref{L:nonnegative},
then $\max\{i\mid F_i^\omega\neq 0\}=-\innerprod{\gamma,\beta}$.
\end{lemma}

\begin{theorem} Assume that $\hom_{K_m}(\gamma,\beta)=0$. If a line weight $\omega$ and its dual satisfy the condition in Lemma \ref{L:nonnegative}, then the complex $F_\bullet^\omega$ resolves a maximal Cohen-Macaulay module supported on $\Rep_\gha(K_m)$.
\end{theorem}

In all examples below, we use a computer program based on \cite{F} to calculate the Kronecker coefficients. We will use the shorthand $g_{\mu,\nu}^{\lambda}(\mu;\nu;\lambda')$ for a typical summand $S^{\mu}R_1  \otimes S^{\nu}R_2^*  \otimes \big(g_{\mu,\nu}^{\lambda} S^{\lambda'}R_{12}\big) \otimes A(-|\lambda|)$. We may wrap several $S^{\lambda'}R_{12}$'s with common $\mu,\nu$ in one pair of parentheses.

\begin{example} Consider $K_3$ with $\alpha=(3,3)$.
There are only three $\gamma$'s up to symmetry such that $\Rep_\gha(K_3)$ is nontrivial. They are $(3,2),(2,1),$ and $(2,2)$.
$(3,2)$ is uninteresting because it is a usual determinantal variety.
We found that for $\gamma=(2,2)$, the terms of $F_\bullet$ are
$$F_0 = (0;0;0)\oplus(1^3;1^3;2,1)\oplus (2,1;1^3;1^3),\ F_1 = (2,1^2;2,1^2;2,1^2),\ F_2 = (2^3;4,1^2;2^3).$$
As an illustration, let us compute $F_1$ from Proposition \ref{P:Fi}.
We first find all $(t_1,t_2,\mu,\nu)$ such that
$$0\leqslant t_1\leqslant \gamma_1,0\leqslant t_2\leqslant \beta_2;\ \mu\in P(\gamma_1,\beta_1,t_1,0),\nu\in P(\beta_2,\gamma_2,t_2,0).$$
We get $t_1=t_2=1, \mu=(3,1) \text{ or } (4),\nu=(4)$.
Then we compute the Kronecker coefficients $g_{\mu,\nu}^{\lambda}$ for each solution $(t_1,t_2,\mu,\nu)$.
Since the partition $(4)$ corresponds to the trivial representation of $S_4$,
the only nonzero $g_{\mu,\nu}^{\lambda}$ we can get are
$g_{(4),(4)}^{(4)}=g_{(3,1),(4)}^{(3,1)}=1$.
But $S^{(4)'}(R_{12})=S^{(1^4)}(k^3)$ vanishes.
So we only apply the formula \eqref{eq:circ1} and \eqref{eq:circ2} to $\mu=(3,1),\nu=(4)$, and get $\mu^\circ=(2,1^2),\nu^\circ=(2,1^2)$.

Using the $G$-equivariant property, it is not hard to make the differentials explicit.
For example, the differential $d_2:F_2\to F_1$ are induced by multiplying
$$(1^2; 2; 1^2)\subset S^2(k^3\otimes k^3\otimes k^3)\subset A.$$
We denote this by $F_1^{}\xleftarrow{\cdot (1^2; 2; 1^2)} F_2^{}.$
The complex $F_\bullet$ is
$$F_0^{}\xleftarrow{\cdot(2,1^2;2,1^2;2,1^2)\oplus(1;1;1)\oplus(1;1;1)} F_1^{}\xleftarrow{\cdot (1^2; 2; 1^2)} F_2^{}.$$
We can check using Proposition \ref{P:birational} that $q$ is birational. So $F_\bullet$ is the minimal free resolution of the normalization of $\Rep_\gha(K_3)$.
With a little effort, we can explicitly identify differentials with matrices in $A$.
In general, finding the matrix representation of differentials is non-trivial. But when the differential is linear, it is always possible \cite{Sa}.
We can easily obtain the set-theoretical defining equations of $\Rep_\gha(K_3)$ from the twisted complex
$$F_0^{(2;1)}=(2;1^2;0)\xleftarrow{\cdot(1;1;1)} F_1^{(2;1)}=(2,1;1^3;1)\xleftarrow{\cdot(2,1;1^3;2,1)} F_2^{(2;1)}=(2^3;2^3;2^2).$$

Now let $\gamma=(2,1)$, then
\begin{align*}
F_0 &= (0;0;0)\oplus(1^2;1^2;1^2)\\
F_1 &= (1^3;1^3;(1^3\oplus 2,1 \oplus 3))\oplus (2,1;2,1;1^3)\oplus (2,1;1^3;2,1) \oplus (1^3;2,1;2,1)\\
F_2 &= (2,1^2;2,1^2;(3,1\oplus 2,1^2\oplus 2^2)) \oplus (2,1^2;3,1;2,1^2) \\
& \oplus (3,1;2,1^2;2,1^2) \oplus(2^3;2^3;(4,1^2\oplus 3^2))\\
F_3 &= (3,1^2;3,1^2;(3,1^2\oplus 2^2,1)) \oplus (4,1^2;2^3;3,2,1)\oplus(2^3;4,1^2;3,2,1) \\
& \oplus (5,1;2^3;2^3)\oplus(2^3;5,1;2^3)\oplus(3,2^2;3,2^2;(4,2,1\oplus 4,3\oplus 3,2^2\oplus 3^2,1))\\
F_4 &= (5,1^2;5,1^2;3,2^2)\oplus(3,2^2;3,2^2;3,2^2) \oplus(4,2^2;4,2^2;(4,2^2\oplus 4,3,1\oplus 3^2,2))\\
& \oplus(4,2^2;3^2,2;(4,3,1\oplus 3,3,2)) \oplus (3^2,2;4,2^2;(4,3,1\oplus 3^2,2))\oplus (3^2,2;3^2,2;(4,2^2\oplus 4^2) )\\
F_5 &= (5,2^2;5,2^2,3^3)\oplus(5,2^2;4,3,2;4,3,2)\oplus(4,3,2;5,2^2;4,3,2) \\
& \oplus(4,3,2;4,3,2;(4,3,2\oplus 4^2,1\oplus 3^3))\\
F_6 &= (5,3,2;5,3,2;4,3^2)\oplus (5,3,2;4^2,2;4^2,2)\oplus (4^2,2;5,3,2;4^2,2)\oplus (4^2,2;4^2,2;4,3^2)\\
F_7 &= (5,4,2;5,4,2;4,4,3)\\
F_8 &= (5^2,2; 5^2,2; 4^3)
\end{align*}
We can check using Proposition \ref{P:birational} that $q$ is birational. So $F_\bullet$ is a minimal free resolution of the normalization.
We find that it is impossible to identify $d_1$ using the $G$-equivariant property only. For example, the last three summands of $F_1$ can map into both summands of $F_0$.
However, if twisted by $(2;1)$, we get
$$F_0^{(2;1)}=(2;1;0)\xleftarrow{\cdot(1;1;1)} F_1^{(2;1)}=(2,1;1^2;1)\xleftarrow{} \cdots$$
We note that both presentations $d_1:F_1\to F_0$ and $d_1^{(2;1)}:F_1^{(2;1)}\to F_0^{(2;1)}$ are uniform in the sense that the formula does not change if we increase the number of arrows.
By an extensive search, we believe that there exists no line weight such that the twisted complex is {\em pure}.
\end{example}

\section{Applications}

\subsection{Codimension 1 cases}
If $\Rep_{\gha}(K_m)$ has codimension one in $\Rep_\alpha(K_m)$, then it corresponds to an irreducible polynomial $\Delta_{\alpha,m}^\gamma$ in $k[\Rep_\alpha(K_m)]$.
It is clear from the representation-theoretic meaning of $\Rep_\gha(K_m)$ that all such polynomials are semi-invariants of $G$, i.e., $\SL(R_1)\times \SL(R_2)\times \SL(R_{12})$-invariant.

\begin{definition} The polynomial $\Delta_{\alpha,m}^\gamma$ is called the {\em hyper-polynomial of quiver type} $(m,\alpha;\gamma)$.
\end{definition}

\begin{proposition} \label{P:det} If $q:Z\to \Rep_\gha(K_m)$ is birational, then the determinant of the complex $F(\mc{V})_\bullet$ is equal to $(\Delta_{\alpha,m}^\gamma)^{\rank \mc{V}}$.
\end{proposition}

\begin{proof} The proof is similar to \cite[Proposition 9.1.3]{W}.
\end{proof}

We refer readers to \cite[Appendix A]{GKZ} for the definition of the determinant of a (based exact) complex.
The most interesting case is when the complex has the Cohen-Macaulay property, i.e., has $F_0$ and $F_1$ only.
In this case, the determinant of the complex becomes the usual determinant.

\begin{example} \label{ex:rigid} A triple $(a_1,a_2,a_3)$ is called quiver-rigid if there is some choice of $i,j,k$ such that $\alpha=(a_i,a_j)$ is a {\em rigid} dimension vector of the $a_k$-arrow Kronecker, which means that $\Rep_\alpha(K_{a_k})$ has a dense orbit for the $\GL_\alpha$-action.
In this case, a necessary and sufficient condition for $\Rep_{\alpha}(K_{a_k})$ having $G$-semi-invariants is that $\alpha$ is a multiple of some real Schur root.
Then there is a unique $G$-semi-invariant, which can be easily constructed using quiver methods \cite{S1}. In this sense, they are not very interesting.

We found that all such triples for $2\leqslant a_1\leqslant a_2\leqslant a_3\leqslant 5$ are
$(2,2,3),(2,2,4),(2,3,4)$ and $(2,4,5)$.
For $K_2, \alpha=(2,3), \gamma=(1,1)$, we have
$$F_0=(0;0;0)\oplus(1^2;1^2;1^2),\ F_1=(2,1;1^3;2,1).$$
For $K_2, \alpha=(3,2), \gamma=(2,1)$ and $K_3, \alpha=(2,2), \gamma=(1,1)$,
their complexes are permutations of $F_\bullet$ on three factors.
If we twist $F_\bullet$ by some weights, we get many determinantal representations of the same hyper-polynomial.
\begin{align*}
&F_0^{(1;0)}=(1;0;0) \xleftarrow{\cdot(2,1;1^3;2,1)} F_1^{(1;0)}=(2^2;1^3;2,1),\\
&F_0^{(0;1)}=(0;1;0) \xleftarrow{\cdot(1^2;1^2;2)} F_1^{(0;1)}=(1^2;1^3;2),\\
&F_0^{(1;1)}=(1;1;0) \xleftarrow{\cdot(1;1;1)} F_1^{(1;1)}=(1^2;1^2;1).
\end{align*}
More generally, for $K_2, \alpha=(n,n+1), \gamma=(1,1)$ we have the degree $n(n+1)$ polynomial
$$F_0^{(1;1)}=(1^{n-1};1;0) \xleftarrow{\cdot(1;1;1)} F_1^{(1;1)}=(1^n;1^2;1).$$
\end{example}

\begin{example} \label{ex:codim1}
In this example, we find all remaining hyper-polynomials of quiver type for $2\leqslant m, \alpha_1, \alpha_2\leqslant 5$ (up to symmetry) using determinantal complexes.
We can easily verify using Proposition \ref{P:birational} that the map $q$ is birational for all cases below.
It is quite surprising that we can find a (non-unique) weight such that the differential is linear.
We give both the untwisted complex and twisted one with linear differential.
\begin{align*}
&\quad K_3, \alpha=(3,4),\gamma=(2,3),\quad \deg(\Delta_{\alpha,m}^\gamma)=24.\\
&F_0=(0;0;0)\oplus (2,1^2;1^4;2,1^2),\ F_1=(2^3;3,1^3;2^3),\\
&F_0^{(2;1)}=(2;1^3;0) \xleftarrow{\cdot(1;1;1)} F_1^{(2;1)}=(2,1;1^4;1). \\
\\
&\quad K_3, \alpha=(5,3),\gamma=(3,2),\quad \deg(\Delta_{\alpha,m}^\gamma)=30.\\
&\ \ (\text{$K_5, \alpha=(3,3),\gamma=(1,2)$ is the same up to symmetry}),\\
&F_{-1}=(1^3;1^3;1^3),\ F_0=(0;0;0)\oplus(1^4;2,1^2;2,1^2),\ F_1=(1^5;3,1^2;3,1^2),\\
&F_0^{(1;1)}=(1^2;1^2;0) \xleftarrow{\cdot(1;1;1)} F_1^{(1;1)}=(1^3;1^3;1).\\
\intertext{We observe that this case can be obtained by applying the reflection functor to the first case. In particular, the property that $\mc{R}^i(q_*\mc{O}_Z)=0, i>0$ is {\em not} preserved under reflection.}
\\
&\quad K_4, \alpha=(4,4),\gamma=(1,2),\quad \deg(\Delta_{\alpha,m}^\gamma)=80.\\
&F_{-1}=(1^4;1^4;2,1^2)\oplus(2,1^2;1^4;1^4),\\ &F_0=(0;0;0)\oplus(2,1^3;2,1^3;2,1^3),\ F_1=(2^4;5,1^3;2^4). \\
&F_0^{(1;2)}=(1^3;2^2;0) \xleftarrow{\cdot(1;1;1)} F_1^{(1;2)}=(1^4;2^2,1;1).\\
\\
&\quad K_5, \alpha=(4,5),\gamma=(1,3),\quad \deg(\Delta_{\alpha,m}^\gamma)=200.\\
& F_{-2}=(1^4;1^4;1^4),\ F_{-1}=(2,1^3;1^5;2,1^3)\oplus(2,1^3;2,1^3;1^5), \\
&F_0=(0;0;0)\oplus (3,1^3;2,1^4;2,1^4),\ F_1=(7,1^3;2^5;2^5),\\
& F_0^{(1;2)}=(1^3;2^3;0)\xleftarrow{\cdot(1;1;1)} F_1^{(1;2)}=(1^4;2^3,1;1).\\
\end{align*}
We note that all four hyper-polynomials except for the second one are not hyperdeterminants defined in \cite{GKZ}.
We can see this simply by degree consideration. The hyperdeterminants for $3\times 3\times 4, 4\times 4\times 4,$ and $4\times 5\times 5$ hypermatrices have degree $48, 272$ and $880$ respectively.
\end{example}

\subsection{Kronecker Coefficients}

From Theorem \ref{T:codim}.(1) and Proposition \ref{P:Fi}, we get an interesting result on vanishing of the Kronecker coefficients.
\begin{proposition} \label{P:KRc}
For $\mu\in P(\gamma_1,\beta_1,t_1,0), \nu\in P(\beta_2,\gamma_2,t_2,0)$ with $|\mu|=|\nu|>\beta_1 t_1+\gamma_2 t_2+\ext_{K_m}(\gamma,\beta)$,
we have that $g_{\mu,\nu}^\lambda$ vanishes if $\lambda_1\leqslant m.$
\end{proposition}

This result is sharp in the sense that there are $\mu\in P(\gamma_1,\beta_1,t_1,0), \nu\in P(\beta_2,\gamma_2,t_2,0)$ with $|\mu|=|\nu|=\beta_1 t_1+\gamma_2 t_2+\ext_{K_m}(\gamma,\beta)$ such that $g_{\mu,\nu}^\lambda\neq 0$ for some $\lambda$ with $\lambda_1\leqslant m.$
From Theorem \ref{T:dual}, we obtain a dual version

\begin{proposition} \label{P:KRc_dual}
For $\mu\in P(\gamma_1,\beta_1,t_1,m\beta_2-\alpha_1), \nu\in P(\beta_2,\gamma_2,t_2,m\gamma_1-\alpha_2)$ with $|\mu|=|\nu|<\beta_1 t_1+\gamma_2 t_2-\hom_{K_m}(\gamma,\beta)$,
we have that $g_{\mu,\nu}^\lambda$ vanishes if $\lambda_1\leqslant m.$
\end{proposition}

\begin{proof}[Proof of Theorem \ref{T:intro2}] We observe from the proof of Theorem \ref{T:codim}.(1) that the statement actually holds for the twisted complex $F_\bullet(\mc{V})$, where $\mc{V}$ is any ample line bundle on $\Gr\sm{\alpha\\ \gamma}$. This is because the Grauert-Riemenschneider vanishing theorem \cite[Theorem 1.2.28]{W} holds for the canonical sheaf tensoring with any ample line bundle. An ample line bundle corresponds to negative $w_1$ and $w_2$. So the above two propositions generalize to Theorem \ref{T:intro2}.
\end{proof}

If $\alpha=(n,nm), \gamma=(n,\gamma_2)$ or $\alpha=(nm,n),\gamma=(\gamma_1,0)$, then $\Rep_{\gha}(K_m)$ is the usual determinantal variety of rank $\gamma_2$ maps from $k^{mn}$ to itself. In particular, it is Gorenstein \cite[Corollary 6.1.5]{W}. We conjecture that this is actually an ``if and only if" statement for $\Rep_{\gha}(K_m)$ being Gorenstein.

\begin{proposition} $g_{lm^n,mn^l}^{m^{nl}}=1$ for any $l,m,n\in\mathbb{N}$.
\end{proposition}

\begin{proof} Let $\alpha=(n,nm), \gamma=(n,\gamma_2)$ with $l:=nm-\gamma_2>0$, then $\innerprod{\gamma,\beta}=\ext_{K_m}(\gamma,\beta)=l^2>0$.
According to the above remark, the last term $F_{l^2}$ has rank $1$, so by Proposition \ref{P:Fi} that there is only one solution for $t_1,t_2,\lambda,\mu,\nu$ with
\begin{equation} \label{eq:t1t2} 0\leqslant t_1\leqslant n, 0\leqslant t_2\leqslant l,\mu\in P(0,t_1),\nu\in P(\gamma_2,t_2)
\end{equation}
such that $l^2+0t_1+\gamma_2t_2=|\lambda|=|\mu|=|\nu|$
with $\lambda',\mu^\circ,\nu^\circ$ having exactly $m,n,nm$ equal parts, and $g_{\mu,\nu}^\lambda=1$.

We claim that $t_2=l,\lambda=m^{nl},\mu=lm^n,\nu={mn}^l$ is the only solution. It is clear from \eqref{eq:t1t2} that $t_2=l$, and thus $|\lambda|=lmn, \lambda'=nl^m,\mu^\circ=lm^n,\nu^\circ=l^{mn}.$
Applying the inverse of Bott's algorithm, we get $\mu=lm^n, \nu=mn^l$.
\end{proof}

\section{Generalization to the tensor setting} \label{S:generalization}
We have a straight-forward generalization from the quiver setting to the tensor setting.
We viewed the representation space of the Kronecker quiver as the triple tensor $R_1^*\otimes R_2\otimes R_{12}^*$.
We may just consider the tensor $R_\alpha:=R_1^*\otimes R_2^*\otimes R_3^*$.
We call $\alpha=(\dim R_1,\dim R_2,\dim R_3)$ the dimension vector of the tensor.
Now we consider the product of Grassmannians $\Gr\sm{\alpha\\ \gamma}:=\prod_{i=1}^3 \Gr\sm{\alpha_i\\ \gamma_i}$.
We replace the vector bundle $\mc{E}$ in Section \ref{S:Kron} by
$\mc{S}_1^*\otimes \mc{S}_2^*\otimes \mc{S}_3^*$, where each $\mc{S}_i$ is the (pullback) of the universal subbundle of $\Gr\sm{\alpha_i\\ \gamma_i}$.
We have an induced vector bundle epimorphism $\Gr\sm{\alpha\\ \gamma}\times R_\alpha \twoheadrightarrow \mc{E}$.
Let $\mc{Z}$ be the kernel of the vector bundle epimorphism, and $Z$ be its total space.
We denote by $q$ the projection $Z\to R_\alpha$, and set $R_\gca:=q(Z)$ be the scheme-theoretical image.
Since $Z$ is integral and $q$ is projective, $R_\gca$ is integral and closed.
From now on, we use $q$ to denote the projection $Z\to R_\gca$.
Let $\beta=\alpha-\gamma$, and $\op{h}(\gamma,\beta)$ be the dimension of generic fibre of $q$ and $\op{e}(\gamma,\beta)$ be the codimension of $R_\gca$ in $R_\alpha$, then $$\Innerprod{\gamma,\beta}:=\op{h}(\gamma,\beta)-\op{e}(\gamma,\beta)=\innerprod{\gamma,\beta}_0-\gamma_1\gamma_2\gamma_3.$$
Unfortunately, we do not have an algorithm to compute $\op{e}(\gamma,\beta)$. We also do not have a criterion for the birationality of $q$.

We consider the complex $F_\bullet$ as we did in the quiver setting.
We have analogues of Theorem \ref{T:main}, \ref{T:codim}, and \ref{T:dual}.
More generally, we can twist $F_\bullet$ by a vector bundle $\mc{V}$.
For fixed {\em weight} $\omega=(\omega_1;\omega_2;\omega_3)\in \prod_{i=1}^3\mathbb{Z}^{\beta_i}$, we put the vector bundle $\mc{V}:=\bigotimes_{i=1}^3 S^{\omega_i}\mc{Q}_i$.
The $i$th term of the twisted complex $F_\bullet^\omega$ is
$$F_i^\omega =\bigoplus_{j\geqslant 0} H^j(\Gr\sm{\alpha\\ \gamma}; \bigwedge^{i+j} \mc{E}^* \otimes \bigotimes_{i=1}^3 S^{\omega_i}\mc{Q}_i)\otimes A(-i-j).$$
If $\omega=(w_1^{\beta_1};w_2^{\beta_2};w_3^{\beta_3})$ for $w_i\in\mathbb{Z}$, then $\mc{V}$ is a line bundle .
We simply write $(w_1;w_2;w_3)$ for $\omega$. The proof of the following proposition is almost the same as Proposition \ref{P:Fi}.

\begin{proposition} \label{P:TFi} The $i$th term of the complex $F^\omega_\bullet$ is given by
\begin{align*}\bigoplus_{\substack{|\lambda_k|=i+\sum_k \ell(\sigma_k)}} g_{\lambda_2,\lambda_3}^{\lambda_1'} \big( \bigotimes_{k=1,2,3} S^{\sigma_k(\omega_k,\lambda_k)}R_k \big) \otimes A(-|\lambda|).
\end{align*}
In particular, if $\omega=(w_1;w_2;w_3)$ is a line weight, then $F^\omega_i$ is given by
\begin{align*}\bigoplus_{\substack{0\leqslant t_i\leqslant \gamma_i}}
\bigoplus_{\substack{\lambda_k\in P(\gamma_i,\beta_i,t_i,w_i)\\ |\lambda_k|=i+\sum \beta_i t_i}}
g_{\lambda_2,\lambda_3}^{\lambda_1'} \big( \bigotimes_{k=1,2,3}  S^{\lambda_k^\circ}R_k\big)  \otimes A(-|\lambda_1|),
\end{align*}
where
$$\lambda_i^\circ:= \sigma(t_i)\circ (w^{\beta_i},\lambda)=\big( (\lambda_{i})_1-\beta_i,\dots,(\lambda_{i})_{t_i}-\beta_i,(t_i+w_i)^{\beta_i},(\lambda_{i})_{t_i+1},\dots,(\lambda_{i})_{\gamma_i} \big).$$
Assume that $F_\bullet^\omega$ has no negative degree terms. Then the maximal minors of $d_1:F_1^\omega\to F_0^\omega$ defines $R_\gca$ set-theoretically, and the annihilator of the module $\Coker d_1$ defines $R_\gca$ scheme-theoretically.
\end{proposition}

%

Now for each weight $\omega=(\omega_1;\omega_2;\omega_3)$, we consider the dual weight $$\omega^\vee=(\gamma_2\gamma_3-\alpha_1-\omega_1;\gamma_1\gamma_3-\alpha_2-\omega_2,\gamma_1\gamma_2-\alpha_3-\omega_3).$$
It is easy to verify that it corresponds to the dual bundle of Theorem \ref{T:dual}, so
$$F^{\omega^\vee}_\bullet=(F^{\omega}_\bullet)^*[\Innerprod{\gamma,\beta}].$$

Analogous to theorem \ref{T:intro2}, we get a vanishing condition for the Kronecker coefficients from the above proposition and Theorem \ref{T:codim}.(1).
Since $g_{\mu,\nu}^\lambda$ is in fact invariant under any permutation of $\lambda,\mu,\nu$, we will write $g_{\lambda,\mu,\nu}$ instead of $g_{\mu,\nu}^\lambda$.
\begin{corollary}  \label{C:Kron}  Let $w_i$ be non-positive numbers. \begin{enumerate}
\item For $\lambda_i\in P(\gamma_i,\beta_i,t_i,w_i)$ with $|\lambda_i|>\sum_i \beta_i t_i + \op{e}(\gamma,\beta)$,
we have that $g_{\lambda_1,\lambda_2,\lambda_3}$ vanishes if $(\lambda_i)_1\leqslant \gamma_i$ for some $i$. \\
\item For $\lambda_i\in P(\gamma_i,\beta_i,t_i,\gamma_j\gamma_k-\alpha_i-w_i)$ with $|\lambda_i|<\sum_i\beta_i t_i-\op{h}(\gamma,\beta)$,
we have that $g_{\lambda_1,\lambda_2,\lambda_3}$ vanishes if $(\lambda_i)_1\leqslant \gamma_i$ for some $i$.
\end{enumerate}
\end{corollary}

The proof of the following lemma is similar to that of Lemma \ref{L:nonnegative}, so we leave it for readers.
\begin{lemma} \label{L:Tnonnegative} {\ }\begin{enumerate}
\item $F_\bullet^\omega$ has no term in negative degree if the line weight $\omega=(w_1;w_2;w_3)$ satisfies
$\sum_{i=1}^3(2\beta_i-w_i)^2<12$, or $\sum_{i=1}^3(2\beta_i-w_i)^2=12$ with any of the following
$(1).\ \beta_i$'s are not all equal; $(2).\ w_i$'s are not all equal; $(3).\ \sum_{j\neq i} w_j>\alpha_i-3$ for some $i$.\\
\item If $\omega^\vee$ satisfies the above condition,
then $\max\{i\mid F_i^\omega\neq 0\}=-\Innerprod{\gamma,\beta}$.
\end{enumerate}
\end{lemma}

%

\begin{proposition} \label{P:TMCM} Assume that $\op{h}(\gamma,\beta)=0$. If a line weight $\omega$ and its dual satisfy the conditions in Lemma \ref{L:Tnonnegative}, then the complex $F_\bullet^\omega$ resolves a maximal Cohen-Macaulay module supported on $R_\gca$.
\end{proposition}

If $R_{\gca}$ has codimension one in $R_\alpha$, then it corresponds to an irreducible polynomial $\Delta_{\alpha}^\gamma$ in $k[R_\alpha]$.
Since $q(Z)$ is $G$-stable, all such polynomials are semi-invariants of $G$.

\begin{definition} The polynomial $\Delta_{\alpha}^\gamma$ is called the hyper-polynomial of type $(\alpha;\gamma)$.
\end{definition}

\begin{proposition} If $q$ is birational, then the determinant of the complex $F(\mc{V})_\bullet$ is equal to $(\Delta_{\alpha}^\gamma)^{\rank \mc{V}}$.
\end{proposition}

\begin{example} \label{ex:Tcodim1}
In this example, we find all hyper-polynomials of type $(\alpha,\gamma)$ up to some powers for $2\leqslant  \alpha_i \leqslant 5$ using determinantal complexes.
In contrast to the quiver type, we cannot find a weight such that the differential is linear for the two non-trivial cases below. In these cases the hyper-polynomials are not completely explicit.

\begin{align*}
&\quad \alpha=(3,4,5),\gamma=(2,3,2),\quad \deg(\Delta_\alpha^\gamma)=240.\\
&F_{-1}=(2,1^2;1^4;1^4),\\
&F_0=(0;0;0)\oplus (3,1^2;2,1^3;1^5)\oplus (5,3,2; 3^2,2^2;2^5),\ F_1=(5^2,2;3^4;3^2,2^3),\\
&F_0^{(2;3;-1)}=(2^2,1;3,2,1;0) \xleftarrow{\cdot(1;1;1)\oplus(3,1^2;3,1^2;1^5)} F_1^{(2;3;-1)}=(2^3;3,2^2,1)\oplus (4,3^2;3^3,2;1^5).\\
\\
&\quad \alpha=(4,4,4),\gamma=(2,2,3),\quad \deg(\Delta_\alpha^\gamma)=560.\\
&F_{-2}=(2^4;2^4;3^2,2),\\ &F_{-1}=(1^4;1^4;(1^4\oplus 2,1^2))\oplus(1^4;2,1^2;1^4)\oplus(2,1^2;1^4;1^4)\oplus (3,2^3;3,2^3;(3,2^3\oplus 3^2,2,1)),\\ &F_0=(0;0;0)\oplus(2,1^3;2,1^3;2,1^3)\oplus(3^2,2^2;4,2^3;3^2,2^2)\oplus(4,2^3;3^2,2^2;3^2,2^2),\\ & F_1=(4^2,2^2;4^2,2^2;3^4),\\
&F_0^{(2;0;1)}=(2^2;0;1) \xleftarrow{\cdot(2^2;1^4;2^2)\oplus (3^2,1^2;2^4;2^4)} F_1^{(2;0;1)}=(2^4;1^4;2^2,1)\oplus(4,3^2,2;2^4;3,2^3).\\
\end{align*}
It turns out the rest of the hyper-polynomials are in fact of quiver types. For the first three, this follows from the remark of Example \ref{ex:rigid}.
The conclusion on the rest is based on explicit computation. However, none of $q:Z\to R_\gca$ below is finite.
\begin{align*}
&\quad \alpha=(2,3,4),\gamma=(1,2,3)\quad (\text{same as } K_2, \alpha=(3,4),\gamma=(1,1)),\\
&\quad \alpha=(2,4,5),\gamma=(1,2,4)\quad (\text{same as } K_2, \alpha=(4,5),\gamma=(1,1)),\\
&\quad \alpha=(2,4,5),\gamma=(1,3,3)\quad (\text{same as } K_2, \alpha=(4,5),\gamma=(1,1)),\\
&\quad \alpha=(3,3,5),\gamma=(1,2,4)\quad (\text{same as } K_3, \alpha=(5,3),\gamma=(3,2)),\\
&\quad \alpha=(4,4,4),\gamma=(1,3,3)\quad (\text{same as } K_4, \alpha=(4,4),\gamma=(2,3)),\\
&\quad \alpha=(4,5,5),\gamma=(1,3,4)\quad (\text{same as } K_5, \alpha=(4,5),\gamma=(1,3)).
\end{align*}
\end{example}

\section*{Acknowledgement}
The author would like to thank Professor Jerzy Weyman for carefully reading the manuscript.

\bibliographystyle{amsplain}

\end{document}